\documentclass{ifacconf}
\usepackage{natbib}

\usepackage{graphicx}
\graphicspath{{./fig/}}
\usepackage[usenames,dvipsnames]{color}
\usepackage{epstopdf}
\usepackage{comment}
\usepackage[toc,acronym,nonumberlist]{glossaries}

\newacronym{pr}{PR}{Positive Real}
\newacronym{spr}{SPR}{Strictly Positive Real}
\newacronym{kyp}{KYP}{Kalman-Yakobovich-Popov}
	
\usepackage{savesym}
\savesymbol{theoremstyle}
\usepackage{amsmath,amsthm,amssymb,mathrsfs,dsfont}
\restoresymbol{AMS}{theoremstyle} 

\newtheorem{theorem}{Theorem}[section]

\newtheorem{definition}[theorem]{Definition}

\newtheorem{remark}[]{Remark}
\newtheorem{assumption}[theorem]{Assumption}

\newcommand*\diag[0]{\mbox{diag}}
  
\usepackage{ifthen}
\newboolean{showcomments}
\setboolean{showcomments}{true}

\newcommand{\enrique}[1]{\ifthenelse{\boolean{showcomments}}
{\textcolor{Blue}{\bf(Enrique says: #1)}}{}}
\newcommand{\addcite}[0]{\ifthenelse{\boolean{showcomments}}
{\textcolor{Purple}{~(add cite(s))}}{}}
\newcommand{\addcites}[0]{\ifthenelse{\boolean{showcomments}}
{\textcolor{Purple}{~(add cite(s)) }}{}}
\newcommand{\todo}[1]{ {\ifthenelse{\boolean{showcomments}}
{\bf\color{red}$\fbox{To do:}$} #1\\}}
\newcommand\oprocendsymbol{\hbox{$\square$}}
\newcommand\oprocend{\relax\ifmmode\else\unskip\hfill\fi\oprocendsymbol}

\usepackage{mathtools}
\usepackage{enumerate}
\newcommand{\richard}[1]{\ifthenelse{\boolean{showcomments}}
{\textcolor{Maroon}{\bf(Richard says: #1)}}{}}

\newcommand{\abs}[1]{\ensuremath{\left\vert#1\right\vert}}

\newcommand{\cfunof}[1]{\ensuremath{\left\{#1\right\}}}

\newcommand{\C}{\ensuremath{\mathbb{C}}}

\newcommand{\funof}[1]{\ensuremath{\left(#1\right)}}

\newcommand{\jw}{\ensuremath{\left(j\omega\right)}}
\newcommand{\jwo}{\ensuremath{\left(j\omega_0\right)}}

\newcommand{\norm}[1]{\ensuremath{\left\Vert #1 \right\Vert}}

\renewcommand{\H}{\ensuremath{\mathbf{H}_{\infty}}}

\newcommand{\RH}{\ensuremath{\mathbf{H}_{\infty}}}

\newcommand{\R}{\ensuremath{\mathbb{R}}}

\newcommand{\s}{\ensuremath{\left(s\right)}}
\newcommand{\sqfunof}[1]{\ensuremath{\left[#1\right]}}

\begin{document}
\begin{frontmatter}

\title{Decentralized Robust Inverter-based Control in Power Systems\thanksref{footnoteinfo}} 
%

\thanks[footnoteinfo]{
This work was supported by NSF CPS grant CNS 1544771, Johns Hopkins E2SHI Seed Grant, and Johns Hopkins WSE startup funds.}

\author[Richard]{Richard Pates} 
\author[Enrique]{Enrique Mallada} 

\address[Richard]{Lund University, 
    Lund, Sweden (e-mail: richard.pates@control.lth.se).}
\address[Enrique]{Johns Hopkins University, 
   Baltimore, MD 21218 USA (e-mail: mallada@jhu.edu)}

\begin{abstract}                
This paper develops a novel framework for power system stability analysis, that allows for the decentralized design of inverter based controllers. The method requires that each individual inverter satisfies a standard $H^\infty{}$ design requirement. Critically each requirement depends only on the dynamics of the components and inverters at each individual bus, and the aggregate susceptance of the transmission lines connected to it. The method is both robust to network and delay uncertainties, as well as heterogeneous network components, and when no network information is available it reduces to the standard decentralized passivity sufficient condition for stability. We illustrate the novelty and strength of our approach by studying the design of inverter-based control laws in the presence of delays. 
\end{abstract}



\end{frontmatter}


\section{Introduction}

The composition of the electric gird is in state of flux~\citep{Milligan:2015ju}. 
Motivated by the need of reducing carbon emissions, conventional synchronous combustion generators, with relatively large inertia, are being replaced with renewable energy sources with little (wind) or no inertia (solar) at all~\citep{Winter:2015dy}. 
Alongside, the steady increase of power electronics in the demand side is gradually diminishing the load sensitivity to frequency variations~\citep{WoodWollenberg1996}.
As a result, rapid frequency fluctuations are becoming a major source of concern for several grid operators~\citep{Boemer:2010wa,Kirby:2005uy}. Besides increasing the risk of frequency instabilities, this dynamic degradation also places limits on the total amount of renewable generation that can sustained by the grid. Ireland, for instance, is already resourcing to wind curtailment --whenever wind becomes larger than $50\%$ of existing demand-- in order to preserve the grid stability.

One solution that has been proposed to mitigate this degradation is to use inverter-based generation to mimic synchronous generator behavior, i.e. implement virtual inertia~\citep{Driesen:ft,Driesen:fta}. However, while virtual inertia can indeed mitigate this degradation, it is unclear whether that particular choice of control is the most suitable for this task. 
On the one hand, unlike generator dynamics that set the grid frequency, virtual inertia controllers estimate the grid frequency and its derivative using noisy and delayed measurements.
On the other hand, inverter-based control can be significantly faster than conventional generators. Thus using inverters to mimic generators behavior does not take advantage of their full potential. 

Recently, a novel dynamic droop control (iDroop)~\citep{m2016cdc} has been proposed as an alternative to virtual inertia that seeks to exploit the added flexibility present in inverters. Unlike virtual inertia that is sensitive to noisy measurements (it has unbounded $\mathcal H_2$ norm~\citep{m2016cdc}), iDroop can improve the dynamic performance without unbounded noise amplification. 
However, as more sophisticated controllers such as iDroop are deployed,  the dynamics of the power grid become more complex and uncertain, which makes the application of direct stability methods harder.
The challenge is therefore to design an inverter control architecture that takes advantage of the added flexibility, while providing stability guarantees. Such an architecture must take into account the effect of delays and measurement noise in the design. It must be robust to unexpected changes in the network topology. And must provide a ``plug-and-play'' functionality by yielding decentralized --yet not conservative-- stability certificates.

In this paper we leverage classical stability tools for the Lur'e problem~\citep{BW65} to develop a novel analysis framework for power systems that allows a decentralized design of inverter controllers that are robust to network changes, delay uncertainties, and heterogeneous network components. 
More precisely, by modeling the power system dynamics as the feedback interconnection of input-out bus dynamics and network dynamics (Section \ref{sec:network-model}), we derive a decentralized stability condition that depends only on the individual bus dynamics and the aggregate susceptance of the transmission lines connected to it (Section \ref{sec:stability-criterion}). When no network information is available, our conditions reduces to the standard decentralized passivity sufficient condition for stability. 
We illustrate the novelty and strength of our analysis framework by studying the design of inverter-based control laws in the presence of delays (Section \ref{sec:analysis-idroop}).


\section{Input-output Power Network Model}\label{sec:network-model}

In this section we describe the input-output representation of the power grid used to derive our decentralized stability results. We use $i\in V:=\{1\dots,n\}$ to denote the $i$th bus in the network and the unordered pair $\{i,j\}\in E$ to denote each transmission line. 

As mentioned in the previous section, we model the power network as the feedback interconnection of two systems, $P:=\diag\funof{p_i,i\in{}V}$ and $N$ shown in Figure \ref{fig:1}. Each subsystem $p_i$ denotes the $i$th \emph{bus dynamics}, and 
maps $u_{P,i}\mapsto y_{P,i}$, where $u_{P,i}$ denotes the $i$th bus exogenous real power injection, i.e., the real power incoming to bus $i$ from other parts of the network or due to unmodeled bus elements. The output $y_{P,i}:=\omega_i$ denotes the bus frequency deviation from steady state.
Similarly, the system $N$ denotes the \emph{network dynamics}, which maps the vector of system frequencies $u_N:=\omega = (\omega_i,\,i\in V)$ to the vector of electric power network demand $y_{N}= (y_{N,i},\,i\in V)$, i.e. $y_{N,i}$ is the total electric power at bus $i$ that is being drained by the network. Thus, if we let $d_P = (d_{P,i},\,,i\in V)$ denote the unmodeled bus power injection, then by definition we get $u_{P}=-y_N+d_P$.
\begin{figure}
    \centering
    \includegraphics[height=.45\columnwidth,width=.65\columnwidth]{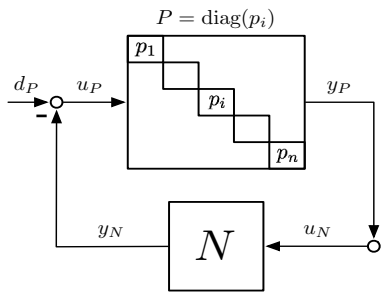}\vspace{-2ex}
    \caption{Input-output Power Network Model}
    \label{fig:1}
\end{figure}

This input-output decomposition provides a general modeling framework for power system dynamics that encompasses several existing models as special cases. For example, it can include the standard swing equations~\citep{Shen:1954eo}, as well as several different levels of details in generator dynamics, including turbine dynamics, and governor dynamics, see e.g., \citep{zmbl2016pscc}.
The main implicit assumption in Figure \ref{fig:1}, which is standard in the literature and well justified in transmission networks~\citep{kundur_power_1994}, is that voltage magnitudes and reactive power flows are constant. However, a similar decomposition could be envisioned in the presence of voltage dynamics and reactive power flows. Such extension will be subject of future research.

In this paper we focus on power system models that satisfy the following assumptions:
\begin{assumption}\label{ass:1}
$P\!=\!\diag\funof{p_1,\ldots{},p_n}$, with $p_i\in\H^{1\times{}1}$.  
\end{assumption}
\begin{assumption}\label{ass:2}
$N=\frac{1}{s}L_B$, where $L_B\in\R^{n\times n}$ is a weighted Laplacian matrix.
\end{assumption}
Here $\H^{m\times n}$ denotes the Hardy space of $m$ by $n$ complex matrix transfer functions analytic on the open right-half  plane ($\mathrm{Re}\cfunof{s} >0$) and bounded on the imaginary axis ($j\R$). We only consider such a general class of transfer functions to model delays. In almost all remaining cases we only use real rational transfer functions.



In the rest of this section we illustrate how different network components can be modeled using this framework as well as the implications of assumptions \ref{ass:1} and \ref{ass:2}. For concreteness, we make specific choices on  the models for generators and loads. We highlight however that our analysis framework can be extended to more complex models provided they satisfy assumptions \ref{ass:1} and \ref{ass:2}.

\subsection{Bus Dynamics}
We model the bus dynamics using the standard swing equations~\citep{Shen:1954eo}. Thus the frequency of each bus $i$ evolves according to 
\begin{equation}\label{eq:bus-swing}
p_i~:\begin{array}{rl}
\quad M_i\dot\omega_i &= -D_i\omega_i + x_i + u_{P,i}\\
y_{P,i}&=\omega_i
\end{array},
\end{equation}
where the state $\omega_i$ represents the frequency deviation from nominal and $x_i$ denotes the power injected by inverter-based generation at bus $i$.
The parameter $M_i\geq0$ denotes the aggregate bus inertia.
For a generator bus $M_i>0$, and $D_i>0$ represents the damping coefficient. For a load bus $M_i=0$ and $D_i>0$ represents the load sensitivity to frequency variations~\citep{BergenHill-1981}.  

We consider linear control laws for the inverter dynamics that  depend solely on the local frequency. Thus, we model the inverter dynamics as a negative feedback law of the form
\begin{equation}\label{eq:control}
\hat{x}_i\s=-c_i\s\hat{\omega}_i\s,
\end{equation}
where $\hat x_i\s$ and $\hat \omega_i\s$ denote the Laplace transform of $x_i(t)$ and $\omega_i(t)$, respectively.

Combining \eqref{eq:bus-swing} and \eqref{eq:control} we get the following input-output representation of the bus dynamics
\begin{equation}\label{eq:bus-dynamics}
p_i\s=\frac{1}{M_is+D_i}\funof{1+\frac{c_i\s}{M_is+D_i}}^{-1}.
\end{equation}
Whenever $c_i(s)=0$,  \eqref{eq:bus-dynamics} represents a bus without inverter control, and by choosing either $M_i>0$ or $M_i=0$, \eqref{eq:bus-dynamics} can model generator or load buses respectively. Thus, \eqref{eq:bus-dynamics} provides a compact and flexible representation of the different bus elements. 
It is important to notice that Assumption \ref{ass:1} is rather mild and only requires that the transfer function \eqref{eq:bus-dynamics} of each bus is stable. 

Finally,  the control law $c_i\s$ defines a general modeling class for inverter-based decentralized controllers. A number of conventional architectures can be written in this form. For example, \emph{virtual inertia} inverters correspond to 
\begin{equation}\label{eq:droop-control}
c_i\s=K_i+K_i^\nu s
\end{equation}
where $K_i\geq0$ and  $K_i^\nu\geq0$ are the  droop and virtual inertia constants, and by setting $K_i^\nu=0$, we recover the standard \emph{droop control}.
Similarly, the \emph{iDroop} dynamic controller is given by
\begin{equation}\label{eq:idroop}
c_i\s= \frac{K^\nu_is+ K_i^\delta K_i}{s+K_i^\delta}
\end{equation}
where $K_i^\nu\geq0$, $K_i^\delta\geq0$ are tunable parameters\citep{m2016cdc}.

\subsection{Network Dynamics}\label{ssec:network-dynamics}

The network dynamics, under Assumption \ref{ass:2},  are given by 
\begin{equation}\label{eq:net-dynamics}
N:~\begin{array}{rl}
    \dot \theta &= u_N\\
    y_N &= L_B\theta
\end{array},
\end{equation}
where, as mentioned before, the matrix $L_B\in\R^{n\times n}$ is the $B_{ij}$-weighted Laplacian matrix that describes how the transmission network couples the dynamics of different buses.  

Thus Assumption \ref{ass:2} is equivalent to the DC power flow approximation, where the parameter $B_{ij}$ usually denotes the susceptance of line $\{i,j\}$, although more generally represents the sensitivity of the power flowing through line $\{i,j\}$ due to changes on the phase difference, i.e., $B_{ij}=\frac{v_iv_j}{x_{ij}}\cos(\theta_i^*-\theta_j^*)$ where $v_i$ and $v_j$ are the (constant) voltage magnitudes, $x_{ij}$ is the line inductance, and $\theta_i^*-\theta_j^*$ is the steady state phase difference between buses $i$ and $j$, see e.g., \citep{zhao2013power}. 

To simplify the exposition, we refer here to $B_{ij}$ as the transmission line susceptance. Therefore, $$[L_B]_{ii}=\sum_{j:\{i,j\}\in E}B_{ij}$$ denotes the aggregate susceptance of the lines connected to bus $i$. As we will soon see in the next section, an upper bound of this value, i.e. $[L_B]_{ii}\leq\frac{1}{\gamma_i}$, can be used to guarantee stability in a decentralized manner.

\subsection{Connection with the Swing Equations}
For sake of clarity we derive here the full state space model of the described models and show how the standard swing equation model fits into our framework. 
Since $u_N=y_P=\omega$ and $u_P = d_P -y_N$, then using \eqref{eq:bus-dynamics} and \eqref{eq:net-dynamics}, the state space representation of the feedback interconnection in Figure \ref{fig:1} amounts to:
\begin{subequations}\label{eq:swing}
\begin{align}
\dot{ \theta_i} &= \omega_i\label{eq:swing-a}\\
M_i\dot{ \omega_i} & = - D_i\omega_i -\!\!\!\sum_{j:\{i,j\}\in E}\!\!\!B_{ij}(\theta_i-\theta_j)  +x_i+d_{P,i}\label{eq:swing-b}.
\end{align}
\end{subequations}
where the inverter-based power injection $x_i$ evolves according to either 
\begin{align}
\quad x_i &= -K_i\omega_i - K_i^\nu\dot\omega_i\\
\intertext{for virtual inertia inverter-based control, or}
\quad \dot x_i &= -K_i^\delta(K_i\omega_i + x_i) -K_i^\nu\dot\omega_i
\end{align}
of the case of iDroop.

Equation \eqref{eq:swing} amounts to the standard swing equations in which $d_{P,i}$ represents the net constant power injection at bus $i$. One interesting observation, and perhaps also a peculiarity of our framework, is that while usually the phase of bus $i$ ($\theta_i$) is considered to be part of the bus dynamics, in our framework this state is part of the network. This sidesteps the need to define, for example, a reference bus to overcome the lack of uniqueness in these angles.






\section{A Scalable Stability Criterion for Power Systems}\label{sec:stability-criterion}

This section consists of two main parts. First we give a generalisation of a classical stability result for single-input-single-output systems that can be applied to the structured feedback interconnection in Figure~\ref{fig:1}. Second, we discuss how it can be given a plug and play interpretation, and used to guide controller design in the context of the electrical power system models from Section~\ref{sec:network-model}.

\subsection{A Generalisation of a Classical Stability Criterion}

\cite{BW65} gives a method for testing stability of a feedback interconnection in the face of an uncertain gain. More specifically it is shown that for a real rational, stable, strictly proper transfer function $p\s$, the feedback interconnection of $p\s$ and $k\in\R$ is stable for all 
\[
0\leq{}k\leq{}\frac{1}{k^*}
\]
if and only if there exists a \gls{spr} $h\s$ such that
\begin{equation}\label{eq:absstb}
h\s\funof{k^*+p\s}
\end{equation}
is \gls{spr} (see \cite{DA96} for this precise statement). The concept of an \gls{spr} transfer function is given by the following standard definition (e.g. \cite{BL+06}):
\begin{definition}
A  (not necessarily proper) transfer functions $g\s$ is said to be \gls{pr} if:
\begin{enumerate}[(i)]
\item $g\s$ is analytic in $\mathrm{Re}\cfunof{s}>0$;
\item $g\s$ is real for all positive real $s$;
\item $\mathrm{Re}\cfunof{g\s}\geq{}0$ for all $\mathrm{Re}\cfunof{s}>0$.
\end{enumerate}
If in addition there exists an $\epsilon>0$ such that $g\funof{s-\epsilon}$ is \gls{pr}, then $g\s$ is said to be \gls{spr}.
\end{definition}

The following theorem extends this result to the feedback interconnection in Figure~\ref{fig:1}, where the Laplacian matrix $L_B$ plays the role of the uncertain gain. Our motivation for trying to extend this result is driven by a desire for decentralised stability conditions. Since `the gain' of a Laplacian matrix is dependent on the network topology, in order to obtain a result that does not require exact knowledge of the network structure \textit{it is necessary} to be able to handle a degree of uncertainty in this gain. The result in \cite{BW65} is then the natural candidate for extension, because it is the strongest possible result of this type in the scalar case.

\begin{theorem}\label{thm:abs}
Let $\gamma_1,\ldots{},\gamma_n,$ be positive constants, and $P\in\RH^{n\times{}n}$ satisfy Assumption~\ref{ass:1}. If there exists an $h\in\RH$ such that $sh$ is \gls{pr}, and for each $i\in\cfunof{1,\ldots{},n}$:
\begin{equation}\label{eq:maintest}
h\s\funof{\frac{\gamma_i}{2}s+p_i\s}
\end{equation}
is \gls{spr}, then for any
\[
N\in\cfunof{\frac{1}{s}L_B:L_B \text{meets Assumption~\ref{ass:2} and $\sqfunof{L_B}_{ii}\leq{}\frac{1}{\gamma_i}$}},
\]
the feedback interconnection of $P$ and $N$ has no poles in the closed right half plane.
\end{theorem}

\begin{proof}
Since $P\in\RH^{n\times{}n}$, the interconnection of $P$ and $N$ is stable if and only if
\[
N\funof{I+PN}^{-1}\in\RH^{n\times{}n}.
\]
We will now show that the conditions of the theorem are sufficient for the above. Let $D^{-1}=\text{diag}\funof{\frac{\gamma_1}{2},\ldots{},\frac{\gamma_n}{2}}$, and define
\[
A=\funof{D^{\frac{1}{2}}L_BD^{\frac{1}{2}}}.
\]
Hence $N=D^{-\frac{1}{2}}\funof{\frac{1}{s}A}D^{-\frac{1}{2}}$, and $A$ is a normalised weighted Laplacian matrix. Factorise $A$ as
\[
A=QXQ^*,
\]
where $X$ is positive definite with eigenvalues $\lambda(X)\leq 1$ (i.e., $I\geq{}X>0$), $Q\in\C^{n\times{}\funof{n-m}}$, $m>0$, $Q^*Q=I$. Hence $N\funof{I+PN}^{-1}$ equals
\[
\begin{aligned}
&D^{-\frac{1}{2}}Q^*XQD^{-\frac{1}{2}}\funof{sI+PD^{-\frac{1}{2}}Q^*XQD^{-\frac{1}{2}}}^{-1},\\
=&D^{-\frac{1}{2}}Q^*X\funof{sI+QD^{-1}PQ^*X}^{-1}QD^{-\frac{1}{2}}.
\end{aligned}
\]
Clearly then it is sufficient to show that
\begin{equation}\label{eq:intintint1}
\funof{sI+QD^{-1}PQ^*X}^{-1}\in\RH^{\funof{n-m}\times{}\funof{n-m}}.
\end{equation}
The above can be immediately recognised as an eigenvalue condition:
\begin{equation}\label{eq:necsuf2}
-s\notin\lambda\funof{Q^*D^{-1}P\s{}QX},\forall{}s\in\bar{\mathbb{C}}_+.
\end{equation}
By Theorem 1.7.6 of \cite{HJ91}, for any $s\in\C$:
\[
\lambda\funof{Q^*D^{-1}P\s{}QX}\subset\text{Co}\funof{\sqfunof{D^{-1}}_{ii}p_i\s,i\in N}\times{}\sqfunof{\epsilon,1}.
\]
Therefore it is sufficient to show that
\begin{equation}\label{eq:convsets}
-s\notin\text{Co}\funof{\sqfunof{D^{-1}}_{ii}p_i\s,i\in n}\times{}(0,1],
\end{equation}
for all $s=\bar{\mathbb{C}}_+$ (observe that since each $p_i$ is bounded it is enough to check that this holds on a sufficiently large, but bounded, subset of $\bar{\mathbb{C}}_+$). This can be checked using the separating hyperplane theorem, applied pointwise in $s$. In particular, \eqref{eq:convsets} holds if and only if there exists a nonzero $h\in\C$, such that $\forall{}i\in\cfunof{i,\ldots{},n}$:
\begin{equation}\label{eq:prtest}
\text{Re}\cfunof
{h\funof{s+k\sqfunof{D^{-1}}_{ii}p_i\s}}>0,\forall{}\;0<k\leq{}1.
\end{equation}
By a very minor adaptation of the argument in Theorem 2 of \cite{BW65}, we can use a \gls{pr} function to define this $h$ pointwise in $s$. From the conditions of the theorem, and the maximum modulus principle, 
\begin{equation}\label{eq:finaltest}
\text{Re}\cfunof{h\s\funof{s+\sqfunof{D^{-1}}_{ii}p_i\s}}>0,\forall{s}\in\bar{\C}_+.
\end{equation}
Since $sh\s$ is \gls{pr}, for all $k^*\geq{}0$, $$\text{Re}\cfunof{h\s\funof{s+\sqfunof{D^{-1}}_{ii}p_i\s}+k^*sh\s}>0.$$ Dividing through by $\funof{1+k^*}$ shows that under these conditions
\[
\text{Re}\cfunof{h\s\funof{s+\frac{1}{1+k^*}\sqfunof{D^{-1}}_{ii}p_i\s}}>0,\forall{}s\in\bar{\mathbb{C}}_+.
\]
Therefore \eqref{eq:prtest} is satisfied on the entire right half plane for the required range of $k$ values. Consequently \eqref{eq:intintint1} is satisfied, and the interconnection is stable.\end{proof}

\subsection{A Plug and Play Interpretation}\label{subsec:plug-and-play}

In order to give Theorem~\ref{thm:abs} a plug and play interpretation, we have to introduce a little conservatism. In direct analogy with the approach pioneered in \cite{LV06}, we propose to make an a-priori choice of the function $h\s$. Observe that once we have done this, Theorem~\ref{thm:abs} suggests the following entirely decentralised procedure for constructing the network:
\begin{enumerate}
    \item For each component, compute the smallest $\gamma_i$ such that \eqref{eq:maintest} is satisfied.
    \item Check that $\frac{1}{\gamma_i}\geq{}\sqfunof{L_B}_{ii}$. If so, the component can be connected to the network.
\end{enumerate}
The interpretation of the above is that it defines a \textit{network protocol}. Observe that the above process works completely independently of the size of the network, and is therefore highly scalable. In addition we have the following appealing properties:

\textit{Robustness:} Suppose we have an uncertain description of our model $p_i$, for example:
\[
p_i\s\in\cfunof{p\s:p\s=p_{i0}\s+\Delta{}\s,\norm{\Delta\s}_\infty<1}.
\]
Then provided we compute the smallest $\gamma_i$ such that \eqref{eq:maintest} is satisfied for all $p_i$ in this set, we can guarantee stability in a way that is robust to this uncertainty. This has not compromised scalability, since this process remains entirely decentralised.

\textit{Controller Design:} Suppose $p_i$ has some controller parameters that can be specified, for example (as in \eqref{eq:control})
\[
p_i\s=\frac{1}{M_is+D_i}\funof{1+\frac{c_i\s}{M_is+D_i}}^{-1},
\]
where $c_i$ can be chosen. Then provided we can design $c_i$ so that $p_i$ satisfies the protocol, the component can be connected to the network. Observe that this is an \textit{entirely decentralised synthesis condition}, and there is no need to redesign the controller in response to buses being added and removed elsewhere in the network. This is an approach to synthesis that specifically addresses the need for scalability.

The above are just observations about the structure of stability tests. The fact that they are based on testing \gls{spr}ness also brings a number of advantages. More specifically, provided the transfer functions in \eqref{eq:maintest} are rational and proper, this condition can be efficiently tested using the \gls{kyp} lemma. Furthermore, the synthesis problem: design $c_i$ such that
\[
h\s\funof{\frac{\gamma_i{}}{2}s+\frac{1}{M_is+D_i}\funof{1+\frac{c_i\s}{M_is+D_i}}^{-1}}\in\text{\gls{spr}};
\]
is equivalent to an $H^\infty$ optimisation problem, with state space solution given in \cite{SKS94}. Finally if $h\s$ has relative degree 1, then \eqref{eq:maintest} can be checked on the imaginary axis (for a rather complete list of the different characterisations of \gls{spr}ness, see \cite{Wen88}). That is \eqref{eq:maintest} is equivalent to
\begin{equation}\label{eq:freqdom}
\mathrm{Re}\cfunof{h\jw\funof{\frac{\gamma_i}{2}j\omega{}+p_i\jw}}>0,\forall\omega\in\R.
\end{equation}
This allows much classical intuition from frequency responses and Nyquist diagrams to brought to bear on the above synthesis and robustness problems.

While obviously having many appealing features, all of the above rests on one crucial decision: the a-priori choice of $h\s$. We make a few observations:
\begin{enumerate}
    \item If an arbitrary choice is made, this approach can be very conservative.
    \item One need not be overly concerned with optimising the choice of $h$ too much, because the degrees of freedom opened up by the synthesis problem can typically be exploited to meet a wide range of different protocols.
    \item There exists an $h$ such that if
    \[
    p_i\in\gls{spr},
    \]
    then \eqref{eq:maintest} will be satisfied. That is standard passivity results for the interconnection in Figure~\ref{fig:1} are a special case of Theorem~\ref{thm:abs}.
\end{enumerate}

We recommend that $h\s$ be chosen by defining a set of `expected models'
\[
\mathcal{P}=\cfunof{\bar{p}_1\s,\ldots{},\bar{p}_m\s},
\]
and then designing $h\s$ so that \eqref{eq:maintest} is satisfied for all the models in this set. Here $\bar{p}_k$ are transfer functions describing the dynamics of `typical' bus models, perhaps obtained by putting in average values of the model parameters. Provided $m$ is not too large, this problem should be resolvable. The idea is that an $h$ that is designed to work for these transfer functions is a sensible `a-priori' choice, since it guarantees that if the actual bus dynamics are of the `expected' type (or can be designed to be close to the expected type), they will satisfy the network protocol.

\section{Scalable Stability Analysis of iDroop}\label{sec:analysis-idroop}

To motivate both the need for Theorem~\ref{thm:abs}, and to illustrate it's application, consider a two bus network. Assume that the buses have the same dynamics, given by
\begin{equation}\label{eq:fb11}
p_i=\frac{1}{s+0.1}\funof{1+\frac{1}{s+0.1}c_i\s}^{-1},
\end{equation}
and the Laplacian matrix describing the transmission network is given by
\[
L_B=\begin{bmatrix}
1&-1\\-1&1
\end{bmatrix}.
\]
We now consider the problem of how to design an iDroop controller subject to delay. That is, how to design
\[
c_i\s=e^{-s\tau_i}\frac{K_i^\nu{}s+ K^\delta_i K}{s+K_i^\delta},
\]
where $\tau_i>0$ is the delay, such that the interconnection is stable. If the delay equals zero, then passivity theory can be easily applied to answer the stability question. This is because \eqref{eq:fb11} can be rearranged as
\[
p_i=\funof{\funof{\frac{1}{s+0.1}}^{-1}+\funof{c_i\s^{-1}}^{-1}}^{-1},
\]
from which we see that $p_i$ is the parallel interconnection of $c_i^{-1}$ and an \gls{spr} transfer function. This means that if $c_i$ is \gls{spr}, so is $p_i$, which in turn implies that the interconnection with the transmission network is stable (the transmission network is a \gls{pr} transfer function, and it is well known that the feedback interconnection of a \gls{pr} and \gls{spr} transfer function is stable). Hence we arrive at the rather surprising conclusion that any possible choice of iDroop controller will lead to a stable interconnection. This becomes even stranger when realise we can use the same argument in networks of any size. Is it really true that any possible interconnection of heterogeneous buses, equipped with arbitrarily chosen iDroop controllers, interconnected through any possible network, is stable?

The answer is, at least in any practical sense, no. To make this absolutely concrete, let us return to the two bus example, and suppose we choose the following parameters for both our iDroop controllers:
\[
K_i^\nu=1,\;K_i^\delta{}=5,\;K_i=30.
\]
A simple Nyquist argument shows that even a small value of $\tau_i$, for example 0.05, will destabilise this two bus network. It is of course not a limitation of iDroop that this can happen; the introduction of delays (and badly designed controllers) can similarly destabilise networks where the buses employ virtual inertia or droop controllers. Nor is it even that surprising, and the authors would claim no originality for the above argument. It does however serve to illustrate the importance of robust design in the network setting. We will now demonstrate that, as claimed in Section~\ref{subsec:plug-and-play}, Theorem~\ref{thm:abs} can be used to do this.

\begin{figure}
    \centering
    \includegraphics[width=.85\columnwidth]{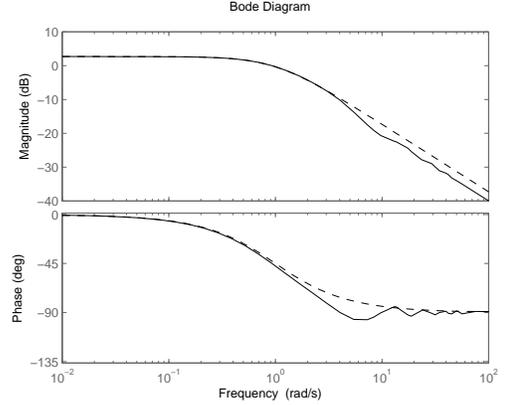}\vspace{-2ex}
    \caption{Bode diagram of $p_i$ (solid curve), and first order approximation of the form in \eqref{eq:approxp}, with $a=1.37,b=1$ (dashed curve).}
    \label{fig:2}
\end{figure}

As discussed in Section~\ref{subsec:plug-and-play}, to obtain a plug and play criterion, we need to make an a-priori choice of $h$. It was also claimed that we need not be overly careful when making this choice, so let us just pick
\[
h\s=\frac{1}{\frac{s}{\omega_0}+1},
\]
and proceed with the robust synthesis of $c_i\s$. Conducting a classical lead-lag design (this is another advantage of the iDroop controller structure) allows us to achieve a $p_i$ which, frequency by frequency, is similar to a transfer function of the form
\begin{equation}\label{eq:approxp}
\frac{a_i}{s+b_i}.
\end{equation}
For a delay of $\tau_i=0.5$ and the model paramters in \eqref{eq:fb11}, this is illustrated in Figure~\ref{fig:2}, where the iDroop controller specified by
\[
K^\nu_i=1.3,\;K_i^\delta{}=8,\;K_i=0.65,
\]
has been chosen. To conclude stability using Theorem~\ref{thm:abs}, we are required to verify that
\begin{equation}\label{eq:newapprox}
\frac{1}{\frac{s}{\omega_0}+1}\funof{\frac{\gamma_i}{2}s+\frac{a_i}{s+b_i}+\Delta_i\s}\in\textrm{\gls{spr}},
\end{equation}
for some $\gamma_i\leq{}1$. In the above $\Delta_i\s$ is the difference between \eqref{eq:approxp} and the actual $p_i$. A simple way to do this is to instead verify the condition
\begin{equation}\label{eq:newnewapprox}
\frac{1}{\frac{s}{\omega_0}+1}\funof{\frac{\gamma_i}{2}s+\frac{a_i}{s+b_i}}-\epsilon_i{}\in\textrm{\gls{pr}}.
\end{equation}
Then provided 
\[
\norm{\frac{1}{\frac{s}{\omega_0}+1}\Delta_i\s}_\infty{}<\epsilon_i{},\;\Longleftrightarrow{}\;\abs{\Delta_i\jw}<\epsilon{}_i\sqrt{1+\frac{\omega^2}{\omega_0^2}},
\]
\eqref{eq:newapprox} is also guaranteed. This relaxation is shown in Figure~\ref{fig:3}. 

\begin{figure}
    \centering
    \includegraphics[width=.85\columnwidth]{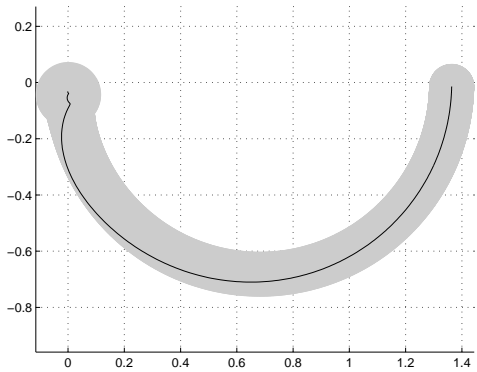}\vspace{-3ex}

    \caption{Nyquist diagram of $p_i$ (black curve), and the uncertainty set on $\Delta_i$ with $\epsilon_i=0.08$ (shaded region). Since the Nyquist diagram lies completely inside the shaded region, satisfying \eqref{eq:newnewapprox} with this $\epsilon_i$ is sufficient for \eqref{eq:newapprox}.}
    \label{fig:3}
\end{figure}

In addition to guaranteeing further levels of robustness to unmodelled dynamics, \eqref{eq:newnewapprox} gives a simple way to characterises any allowable iDroop design in terms of only three parameters: $\funof{a_i,b_i,\epsilon_i{}}$. This is an advantage because the \gls{pr} condition on \eqref{eq:newnewapprox} can be rewritten in terms of inequalities written in terms of these parameters. This means that our network protocol can be understood in terms of (relatively) simple inequalities that only depend on coarse features of our iDroop design. After some simplifications, using the result of e.g.  \cite{FL63}, it can be shown that \eqref{eq:newnewapprox} is equivalent to the following two inequalities:
\[
\begin{aligned}
a_i-\epsilon{}_ib_i&\geq{}0\\
b_i\funof{\gamma{}_i\omega_0-2\epsilon{}_i}-2\epsilon_i\omega_0&\geq{}\ldots{}\\
\ldots{}\frac{\omega_0}{\funof{b_i+\omega_0}}&\funof{\sqrt{a_i-\epsilon{}_ib_i}-\sqrt{b_i\funof{\frac{\gamma_i\omega_0}{2}-\epsilon{}_i}}}^2
\end{aligned}
\]
These give expressions that the design of each individual iDroop controller must satisfy in order to (robustly) satisfy the scalable stability tests. We can easily check them for the two bus example. Substituting in $\funof{a_i,b_i,\epsilon_i}=\funof{1.37,1,0.08}$, and $\omega_0=30$ shows that they are satisfied for any
\[
\gamma_i\geq{}0.18,
\]
which is clearly sufficient to conclude stability. Note that if they were failed, we would then just have to go back and retune the design of our iDroop controller to obtain more favourable parameters (or test \eqref{eq:newapprox} instead). To stress the scalability of this procedure, observe that we did not use the fact that this was a two bus example at any stage. This same local method can be used for designing iDroop controllers for heterogeneous buses in networks of any size, and the above shows that this specific bus can be connected into any possible transmission network provided the local network susceptance is greater than 0.18.

\bibliography{Refs,Refs-em}


\end{document}